\documentclass[12pt,leqno]{amsart}
\usepackage{breqn,xcolor}
\usepackage[english]{babel}
\usepackage{graphicx}
\usepackage{framed}
\usepackage[normalem]{ulem}
\usepackage{amsmath}
\usepackage{amsthm}
\usepackage{amssymb}
\usepackage{amsfonts}
\usepackage{enumerate}
\usepackage[utf8]{inputenc}
\usepackage[top=1 in,bottom=1in, left=1 in, right=1 in]{geometry}
\usepackage{comment}
\RequirePackage{xr-hyper}
\RequirePackage[colorlinks=true]{hyperref}

\newtheorem{theorem}{Theorem}

\newtheorem*{definition}{Definition}

\newtheorem{prop}{Proposition}
\newtheorem{lemma}{Lemma}

\newtheorem{rmk}{Remark}
\theoremstyle{definition}

\numberwithin{equation}{section}
\begin{document}

\title[]
{Nagell-Lutz Theorem for Imaginary Quadratic Fields and class groups of quadratic fields}
\author{Leena Mondal, Amrutha Chalil \and Kalyan Banerjee}
\address[Leena Mondal]{SRM University AP, 
Department of Mathematics, Neerukonda, Amaravati, Andhra Pradesh-522240, India}
\email{leena\_mondal@srmap.edu.in, leenamondal11@gmail.com} 
\address[Amrutha Chalil]{SRM University AP,
Department of Mathematics, Neerukonda, Amaravati, Andhra Pradesh-522240, India}
\email{amrutha\_c@srmap.edu.in,amruthacgangadaran@gmail.com}
\address[Kalyan Banerjee]{SRM University AP,
Department of Mathematics, Neerukonda, Amaravati, Andhra Pradesh-522240, India}
\email{kalyan.ba@srmap.edu.in}


\date{}
 
\maketitle

\begin{abstract}

We prove the Nagell-Lutz theorem for the imaginary quadratic fields of class number one.
 
\end{abstract}

\section{Introduction and statement of results}

Elliptic curves are distinguished class of Diophantine equations, and they play a pivotal role in modern number theory. These curves, defined by equations of the form 
\begin{equation}\label{gnrl eqn}
y^2=x^3+ax+b~~~~~~a,b\in \mathbb Z,
\end{equation} 
are central to the study of integer and rational solutions to polynomial equations (for more details about elliptic curves, we refer to \cite{S-T}). The rich structure of elliptic curves has led to groundbreaking discoveries in areas such as algebraic geometry and modular forms. 

\smallskip

\smallskip

The Nagell-Lutz theorem is a significant result in the study of elliptic curves, particularly when it comes to finding integral solutions on these curves, which are torsion points. Provides a set of criteria to efficiently identify such points, simplifying the search for rational or integer solutions to elliptic curve equations. The theorem is crucial because it provides a finite set of possibilities for checking integer solutions.

\smallskip
 Understanding elliptic curves over imaginary quadratic fields is an important aspect of algebraic number theory, as it connects the theory of elliptic curves with the arithmetic of number fields. These types of question were studied by Kamieny\cite{K}, Kamieny-Najman \cite{KM}, Kenku-Memmose\cite{KM}, Najman\cite{N},\cite{N1} and Thong\cite{T}. In their work,  the  torsion points on elliptic curves over quadratic fields were extensively studied.
 
 Building on this inspiration, we aim to extend these results to apply any imaginary quadratic field with class number one.  Our primary interest lies in studying an elliptic curve defined over imaginary quadratic fields with class number one. This specific case is significant in number theory as it links the properties of elliptic curves with the arithmetic of imaginary quadratic fields that have unique factorization in their ring of integers. That is our interest is to study the elliptic curve
\begin{equation*}
    E:y^2=x^3+Ax+B,~~~~~A,B\in \mathbb Q(\sqrt{D}),
\end{equation*}
where, $$\mathbb Q(\sqrt{D})=\{a+b\sqrt{D}:a,b\in \mathbb Q\}$$ and $$D=-1, -2, -3, -7, -11, -19, -43, -67, -163$$\\
It is of great interest to explore the possible connections between elliptic curves over rationals and those over imaginary quadratic field. For example, as an application of main theorem we obtain a generalization of Soleng's \cite{S} theorem, that starting from certain torsion points on an elliptic curve we can produce elements in the class groups of a family of quadratic fields.

\begin{rmk}
    In the following sections of the paper, we will adopt the notation $\mathcal{O}_K$ to represent the ring of integers of the quadratic field $K=\mathbb Q(\sqrt{D})$, where $D=-1, -2, -3, -7, -11, -19,\\ -43, -67, -163$.
\end{rmk}

Next, we state our main theorem.
\begin{theorem}(\textbf{Extended Nagell-Lutz Theorem})\label{main thm}
Let $E:y^2=x^3+Ax+B$ with $A,B\in \mathcal{O}_K$. If a point $(x,y)\in E$ has finite order which are not $2$-torsion defined over $K$, then both $x~\text{and}~y\in \mathbb Z(\sqrt{D})$.
\end{theorem}

In this theorem it is important to note that the classical part of the extension of Nagell-Lutz tells us that $(x,y)\in \mathcal{O}_K^2$, but here we are improving it for those fields where the rings of integers are not of the form $\mathbb{Z}(\sqrt{D})$.

In the next theorem we assume that the class number of $\mathbb Q(\sqrt{-d})$ is one. 

\begin{theorem}
Given an elliptic curve $y^2=x^3+(-2m+b\sqrt{-d})x$ over $\mathbb Q(\sqrt{-d})$, we have a homomorphism from the group of 2-torsion points on this elliptic curve to the class group of the number field $\mathbb Q(\sqrt{-p^3+4mp})$ for infinitely many primes $p$, not dividing $2m$ and the class group contains a subgroup isomorphic to $\mathbb Z/2\mathbb Z\times \mathbb Z/2\mathbb Z$. Here, $m$ is a positive integer.
\end{theorem}

\section{Proof of the main theorem}

 In this section, we prove some important lemmas that will be used later.

\begin{definition}
    Let $\alpha,\beta\in{\mathcal{O}_K}$, we say that $\alpha$ is divisible by $\beta$, denoted by $\beta\vert \alpha$, if $\beta\neq 0$ and there exists $\gamma\in {\mathcal{O}_K}$ such that $\alpha=\beta \gamma$. 
\end{definition}
 \begin{rmk}
     Let \begin{equation}\label{Elliptic curve1}
         E:y^2=x^3+Ax+B ~~with~~ A,B\in \mathbb Q(\sqrt{D}).
     \end{equation}\\ So,
     \begin{align*}
         A&=\frac{A_1+\sqrt{D}A_2}{A_3+\sqrt{D}A_4},\qquad A_1,A_2,A_3,A_4\in{\mathcal{O}_K}\\
         &=\frac{(A_1A_3+DA_2A_4)+\sqrt{D}(A_2A_3-A_1A_4)}{A_3^2-DA_4^2}\\
         &=\frac{A_1^{'}}{A_2{'}}
     \end{align*}
     where $A_1^{'}=(A_1A_3+DA_2A_4)+\sqrt{D}(A_2A_3-A_1A_4)$ and $A_2^{'}=A_3^2-DA_4^2$ and $A_1^{'}\in {\mathcal{O}}_K, A_{2}^{'}\in \mathbb Z$.
\smallskip
Similarly, $B=\frac{B_{1}^{'}}{B_{2}^{'}}$, $B_{1}^{'}\in{\mathcal{O}_K},B_{2}^{'}\in\mathbb Z$.\\
So \eqref{Elliptic curve1} becomes
\begin{align*}
    y^{2}=x^{3}+\frac{A_{1}^{'}}{A_{2}^{'}}x+\frac{B_{1}^{'}}{B_{2}^{'}}.
\end{align*}\label{ellptic curve2}
Let $D=A_{2}^{'}B_{2}^{'}$, multiplying both sides of \eqref{ellptic curve2} by $D^6$, we will get
\begin{align*}
    (D^3y)^2=(D^2x)^3+D^3A_{1}^{'}B_{2}^{'}(D^2x)+D^5A_{2}^{'}B_{1}^{'}.
\end{align*}\label{ellipticurve3}
Again, let $X=D^2x, Y=D^3y, A^{'}=D^3A_{1}^{'}B_{2}^{'}, B^{'}=D^5A_{2}^{'}B_{1}^{'}$ and then \eqref{ellipticurve3} becomes\\
\begin{equation}
    Y^{2}=X^3+A^{'}X+B^{'}, \qquad A^{'},B^{'}\in{\mathcal{O}_K}.
\end{equation}
Thus, we can assume that our elliptic curve \eqref{Elliptic curve1} has coefficients from $\mathcal{O}_K$.
 \end{rmk}

 \begin{lemma}\label{lemma1}
 Let $E:y^2=x^3+Ax+B,~~ A,B\in{\mathcal{O}_K}$. For any $(x,y)\in E(\mathbb Q(\sqrt{D}))$, $p\vert den(x)$ if and only if $p\vert den(y)$. Here, $den(x)$ denotes the denominator of $x$.   
 \end{lemma}
 \begin{proof}
     Let $x=\frac{x_1}{p^rx_2}$, $y=\frac{y_1}{p^sy_2}$, where $x_i,y_i\in{\mathcal{O}_K}~and~p\nmid x_1x_2, p\nmid y_1y_2$.\\
     If $p\vert den(x)$, then $r>0$ gives
     \begin{align*}
         \frac{y_1^2}{p^{2s}y_2^2}&=\left(\frac{x_1}{p^rx_2}\right)^3+A\left(\frac{x_1}{p^r y_2}\right)+B\\
         &=\frac{x_1^3+Ap^{2r}x_1x_2^2+Bp^{3r}x_2^3}{p^{3r}x_2^3}.
     \end{align*}
    Since $p\nmid x_1$ \\
     \begin{align*}
         p\nmid x_1^3+Ap^{2r}x_1x_2^2+Bp^{3r}x_2^3.
     \end{align*}
    This implies
    \begin{align*}
        2s=3r \implies s>0.
    \end{align*}
    So, $p\vert den(y) $. Converse is also true.\\
    Therefore, $2s=3r\in\mathbb Z\implies s=3q,r=2q,q\in \mathbb Z$.
 \end{proof}
 \begin{rmk}
Let $E:y^2=x^3+Ax+B,~~~~A,B\in{\mathcal{O}_K}$.\\
If $y\neq 0$, we will get\\
\begin{align*}
    \frac{1}{y}=\left(\frac{x}{y}\right)^3+A\left(\frac{x}{y}\right)\left(\frac{1}{y}\right)^2+B\left(\frac{1}{y}\right)^3.
\end{align*}
That is, $E\setminus \{(x,0)\}$ is transformed into $E^{'}:s=t^3+Ats^2+Bs^3$, where $s=\frac{1}{y}$ and $t=\frac{x}{y}$.\\
Therefore we can define a map, \begin{align*}
    \phi:E\setminus \{(x,0)\}\longrightarrow E^{'},
\end{align*} 
by,
\begin{align*}
    (x,y)\longrightarrow(t,s),~~~\infty \longrightarrow(0,0).
\end{align*}
\textbf{Note:} $\phi$ is injective.
 \end{rmk}
 \begin{lemma}\label{lemma2}
     Let $E:y^2=x^3+Ax+B$ with $A,B\in {\mathcal{O}_K}$. Any torsion point $(x,y)\in E(\mathbb Q(\sqrt{D}))$ of order $2$ has $x\in{\mathcal{O}_K}$ and $y=0$.
\end{lemma}
 \begin{proof}
     Let $P=(x,y)$ has order $2$.\\
    In other words,
    \begin{align*}
        2P=\infty,
    \end{align*}
    gives
    \begin{align*}
        P=-P.
    \end{align*}
    \\
    So we can conclude that $y=0$.\\
    Therefore, $x$ is a root of $x^3+Ax+B=0,~~~ A,B\in{\mathcal{O}_K}$.\\
    Since, $P\in E(\mathbb Q(\sqrt{D}))$, we will get $x\in\mathbb Q(\sqrt{D})$.\\
    Let
    \begin{align*}
        x=\frac{a+b\sqrt{D}}{c+d\sqrt{D}}.
    \end{align*}
    Our elliptic curve becomes,
    \begin{align*}
        0=\left(\frac{a+b\sqrt{D}}{c+d\sqrt{D}}\right)^3+A\left(\frac{a+b\sqrt{D}}{c+d\sqrt{D}}\right)+B.
    \end{align*}
    Let $a+b\sqrt{D}=a, c+d\sqrt{D}=b$, then
    \begin{align*}
        0&=\left(\frac{a}{b}\right)^3+A\left(\frac{a}{b}\right)+B\\
        &=a^3+Aab^2+Bb^3.
    \end{align*}
    This implies $a^3\in<b>$, but we have $<a,b>=<c>$ (since we are working on PID and $a$ and $b$ are relatively prime).\\
    So
    \begin{align*}
        <a,b>=<1>.
    \end{align*}
    Again we have
    \begin{align*}
        <b>=<a^3,b>\supseteq <a,b>^3=<1>,
    \end{align*}
    and hence $b$ is a unit in $\mathbb Q(\sqrt{D})$.\\
    Therefore,
    \begin{align*}
        x=\frac{a}{b}\in {\mathcal{O}_K}.
    \end{align*}
 \end{proof}
\begin{definition}
    For any $x\neq 0\in \mathbb Q(\sqrt{D})$, $p$-adic value of $x$ is
    \begin{align*}
        g_p(x)=g_p(\frac{a}{b})=r
    \end{align*}
    where\\
    \begin{align*}
       \frac{a}{b}=p^r\frac{a_1}{b_1},~a,~b,~a_1,~b_1\in{\mathcal{O}_K},~p\nmid a_1b_1. 
    \end{align*}
    
    Therefore by Lemma \ref{lemma1}, we have
    \begin{align*}
        g_p(x)=-2q, g_p(y)=-3q.
    \end{align*}
    Define\\
    \begin{align*}
         E_r=\{(x,y)\in E(\mathbb Q(\sqrt{D})):g_p(x)\leq -2r, g_p(y)\leq -3r\}\bigcup \{\infty\}.
    \end{align*}
   
    \smallskip
    Clearly, $E(\mathbb Q(\sqrt{D}))\supseteq E_1\supseteq E_2\supseteq\ldots$
\end{definition}
\begin{lemma}\label{lemma3}
$(x,y)\in E_r$ if and only if $p^{3r}\vert s$. If $p^{3r}\vert s$ then $p^{r}\vert t$.
\end{lemma}
\begin{proof}
    If $(x,y)\in E_r$, then by definition $g_p(y)\leq-3r$.\\
    So we can write $y=\frac{y_1}{p^{3r}y_2},~~p\nmid y_1$. \\Then $s=\frac{p^{3r}y_2}{y_1}$ $\implies p^{3r}\vert s$.\\ 
    \smallskip
    Conversely, suppose $p^{3r}\vert s$ then $p^{3r}\vert den(y)$. By Lemma \ref{lemma1}, we have $p^{2r}\vert den(x)$. 
    \begin{align*}
         \implies (x,y)\in E_r.
    \end{align*}
    
    \smallskip
    So, if $p^{3r}\vert s\implies (x,y)\in E_r$, then the exact power of $p$ dividing $den(y)$ is $p^{3k}$ for some $k\geq r$. By Lemma \ref{lemma1}, $p^{2k}$ is the exact power of $p$ dividing $den(x)$. Since $t=\frac{x}{y}$, $p^k\vert t$. Thus $p^r\vert t$, as $k\geq r$.
    
\end{proof}
\begin{lemma}\label{lemma 4}
Any vertical line $t=k$, where $k$ is a constant such that $p\vert k$, intersects $E^{'}$ in at most one point $(t,s)$ with $p\vert s$. This line is not tangent at this point of intersection.
\begin{proof}
    Suppose, the line intersects $E^{'}$ at $2$ points $(t_1,s_1),(t_2,s_2)\in E^{'}$, such that $p\vert s_i$ and $s_1\equiv s_2\equiv 0(p)$. Write $s_i=ps_i^{'}$.\\
    \smallskip
    Suppose that
    \begin{align*}
        p^k\vert s_1-s_2  .
    \end{align*}for some $k\geq 1$,\\
    i.e.,
    \begin{align*}
        s_1\equiv s_2(p^k) 
    \end{align*}for some $k\geq 1$.\\
    Consequently,
    \begin{align*}
        s_1^{'}\equiv s_2^{'}(p^{k-1}).
    \end{align*}
    Hence,\\
    \smallskip
    \begin{align*}
        s_1^2\equiv s_2^2(p^{k+1}).
    \end{align*}
    Similarly,
    \begin{align*}
         s_1^3\equiv s_2^3(p^{k+2}).
    \end{align*}
    Therefore,
    \begin{align*}
        s_1=k^3+Aks_2^2+Bs_1^3\equiv k^3+Aks_2^2+Bs_2^3\equiv s_2(p^{k+1}).
    \end{align*}
    By induction $s_1\equiv s_2(p)$ for all $n\geq 1$ and hence $s_1=s_2$.\\
    Slope of the tangent line to the curve is 
    \begin{align*}
        \frac{ds}{dt}&=3t^2+As^2+2Ast\frac{ds}{dt}+3Bs^2\frac{ds}{dt}\\
        &=\frac{3t^2+As^2}{1-2Ast-3Bs^2}.
    \end{align*}
    If the line $t=k$ is tangent to the curve at $(s,t)$, then\\
    \begin{align*}
        1-2Ast-3Bs^2=0
    \end{align*} 
    but $s\equiv t\equiv 0(p)$. Thus
    \begin{align*}
        1-2Ast-3Bs^2\equiv 1.
    \end{align*}
    Therefore, $t=k$ is not tangent to the curve.
    
\end{proof}
    
\end{lemma}
\begin{lemma}\label{lemma 5}
Suppose the line $s=\alpha t+\beta$ intersects the curve at the points $(t_1,s_1)$ and $(t_2,s_2)$ in $E^{'}$. Then
\begin{align*}
    \alpha=\frac{t_2^{2}t_1 t_2+t_1^{2}+As_2^{2}}{1-A(s_1+s_2)t_1-B(s_2^{2}+s_1 s_2+s_1^{2})}.
\end{align*}
\end{lemma}
\begin{proof}
    Suppose $t_1\neq t_2$ then $\alpha=\frac{s_2-s_1}{t_2-t_1}$.\\
    We have $s_i=t_i^3+At_is_i^2+Bs_i^3$,\\
    which gives,
    \begin{align*}
        (s_2-s_1)(1-A(s_1+s_2)t_1-B(s_2^2+s_1s_2+s_1^2))&=(s_2-s_1)-A(s_2^2-s_1^2)t_1-B(s_2^3-s_1^3)\\
        &=(s_2-As_2^2t_2-Bs_2^3)-(s_1-As_1^2t_1-Bs_1^3)+As_2^2(t_2-t_1)\\
        &=t_2^3-t_1^3+As_2^2(t_2-t_1)\\
        &=(t_2-t_1)(t_2^2+t_1 t_2+t_1^2+As_2^2),
    \end{align*}
    so,
    \begin{equation}\label{alpha}
        \frac{s_2-s_1}{t_2-t_1}=\alpha=\frac{t_2^2+t_1 t_2+t_1^2+As_2^2}{1-A(s_1+s_2)t_1-B(s_2^2+s_1s_2+s_1^2)}.\\
   \end{equation}
    If $t_1=t_2$, then by Lemma \ref{lemma 4} both points are identical. By differentiation we will get the tangent line as
    \begin{align*}
        \frac{ds}{dt}=\frac{3t^2+As^2}{1-2Ast-3Bs^2},
    \end{align*}
    which is same as our expression when $t=t_1=t_2$ and $s=s_1=s_2$.
\end{proof}
\begin{lemma}\label{lemma6}
Let $E_r^{'}=\phi(E_r)$ and $E^{'}(\mathbb Q(\sqrt{D}))$. Then $E_r^{'}$ is a subgroup of $E^{'}(\mathbb Q(\sqrt{D}))$.
\end{lemma}
\begin{proof}
    Let $(x,y)\in E_r,\phi(x,y)=(t,s)$.\\
    We have $p^{3r}\vert s$ and $p^r\vert t$ and from \eqref{alpha}, we can conclude that $p^{2r}\vert \alpha$. Also since $\beta=s-\alpha t$, we will get $p^{3r}\vert \beta$. For $p_1,p_2\in E_r^{'}$ and $p_1+p_2=p_3$, let $s=\alpha t+\beta$ be the line passing through $p_1$ and $p_2$, this will implies that
    \begin{align*}
        \alpha t+\beta=t^3+At(\alpha t+\beta)^2+B(\alpha t+\beta)^3,
    \end{align*}
    then
    \begin{align*}
        0=t^3(1+\alpha^2 A+\alpha^3 B)+t^2(2A\alpha \beta+3\alpha^2\beta B)+\ldots
    \end{align*}
    Let $p^*=(t^*,s^*)$ be the new point of intersection. Then $p_3=(t_3,s_3)=-p^*$. Which gives
    \begin{align*}
        t_1+t_2-t_3=t_1+t_2+t^*=-\frac{2A\alpha\beta+2B\alpha^2\beta}{1+\alpha^2A+\alpha^3 B}\equiv 0(p^{5r}).
    \end{align*}
    Since $p^r\vert t_1$ for $i=1,2\implies p^r\vert \pm t_3$. Also $s_3=\alpha t_3+\beta\equiv 0(p^{3r})$. Then by Lemma \ref{lemma1}, $\pm p_3\in E_r^{'}\implies E_r^{'}$ is a subgroup of $E^{'}(\mathbb Q(\sqrt{D}))$.\\
    \end{proof}
    \begin{rmk}
    For $p=(t,s)\in E_r^{'}$, write $t(p)=t$. Then we have, $t(p_1)+t(p_2)=t(p_3)(mod~~p^{5r})$. Suppose $p\in E_1^{'}$ has order $m$.
    Assume that $p\nmid m$, then there exists $r>0$ such that $p\in E_r$ but $p\not\in E_{r+1}$.\\
    If $p\vert m$, then $t(mp)=m.t(p)(mod~~p^{5r})$.\\
    So,\begin{align*}
        0=m.t(p)(mod~~p^{5r})\implies p^{5r}\vert t(p).
    \end{align*}
    Therefore,
    \begin{align*}
        p^{5r}\vert{t(p)},
    \end{align*}
    which is a contradiction, since $5r>r$. So $E_1^{'}$ has no point of finite order. Now if $(x,y)\in E(\mathbb Q(\sqrt{D}))$ is a torsion point, then $\phi(x,y)$ must have finite order in $E_1^{'}$, which contradicts the last lemma. Hence, we have the following proposition.
    \end{rmk}
\begin{prop}\label{proposition}
    If $(x,y)\in E(\mathbb Q(\sqrt{D}))$ is of finite order, which are not a $2$-torsion over $\mathbb{Q}(\sqrt{D})$,  then both $x,y\in{\mathcal{O}_K}$.
\end{prop}
\begin{rmk}
    We have $\mathcal{O}_K=\mathbb Z(\sqrt{D})$ for $D=-1,-2$ and $\mathcal{O}_K=\mathbb Z(\frac{1+\sqrt{D}}{2})$ for $D=-3, -7, -11,-19, -43, -67, -163$.
\end{rmk}
\textbf{Claim:} For $D=-3, -7,-11$, if $x,y$ on $E({\mathbb Q(\sqrt{D}))}$ is a point of finite order which is not a $2$-torsion defined over $\mathbb{Q}(\sqrt{D})$ then $x,y\in \mathbb Z(\sqrt{D})$
\begin{proof}
    Since $x,y\in \mathbb Z(\frac{1+\sqrt{D}}{2})$, we can take $x=a+b(\frac{1+\sqrt{D}}{2})$ and $y=c+d(\frac{1+\sqrt{D}}{2})$. That is,\\
    our \textbf{claim} will be $b$ and $d$ are even.\\
    We have from \eqref{Elliptic curve1},\\
    \begin{align*}
        y^2=x^3+Ax+B,~~~A=A_1+A_2\left(\frac{1+\sqrt{D}}{2}\right), B=B_1+B_2\left(\frac{1+\sqrt{D}}{2}\right),
    \end{align*}
    so,
    \begin{align*}
        \left(c+d(\frac{1+\sqrt{D}}{2})\right)^2&=\left(a+b\left(\frac{1+\sqrt{D}}{2}\right)\right)^3+A\left(a+b(\frac{1+\sqrt{D}}{2})\right)+B\\
        &=\left(a+b\left(\frac{1+\sqrt{D}}{2}\right)\right)^3+\left(A_1+A_2(\frac{1+\sqrt{D}}{2})\right)\left(a+b\left(\frac{1+\sqrt{D}}{2}\right)\right)\\&+\left(B_1+B_2\left(\frac{1+\sqrt{D}}{2}\right)\right).
    \end{align*}
    After expanding everything, we will get\\
    \begin{align*}
        c^2+\frac{d^2}{4}+\frac{d^2D}{4}+\frac{d^2\sqrt{D}}{2}+cd+cd\sqrt{D}&=a^3+\frac{a^2b}{2}+\frac{a^2b\sqrt{D}}{2}+\frac{ab^2}{4}+\frac{ab^2\sqrt{D}}{2}+\frac{ab^2D}{4}+\frac{b^3}{8}+\\&\frac{b^3\sqrt{D}}{4}+\frac{b^3D}{8}+\frac{b^3\sqrt{D}}{8}+\frac{b^3D}{4}+\frac{b^3D^{3/2}}{8}+A_1a+\frac{A_1b}{2}+\\&\frac{A_1b\sqrt{D}}{2}+\frac{A_2a}{2}+\frac{A_2a\sqrt{D}}{2}+\frac{A_2b}{4}+\frac{A_2bD}{4}+\frac{A_2bD^4}{2}.
    \end{align*}
    By comparing the real part, we will get
    \begin{align*}
        c^2+\frac{d^2}{4}+\frac{d^2D}{4}+cd&=a^3+\frac{a^2b}{2}+\frac{ab^2}{4}+\frac{ab^2D}{4}+\frac{b^3}{8}+\frac{b^3D}{4}+A_1a+\frac{A
        _1b}{2}+\frac{A_2a}{2}\\&+\frac{A_2b}{4}+\frac{A_2bD}{4}+\frac{A_2bD}{2},
    \end{align*}
    which will give,\\
    \begin{align*}
        2d^2+2d^2D\equiv 2ab^2+2ab^2D+b^3+b^3D+2b^3D+2A_2b+2A_2bD~(mod~4),
    \end{align*}
    $\implies$
    \begin{align*}
         -4d^2\equiv -4ab^2-2b^3-4A_2b~(mod~4),
    \end{align*}
    $\implies$
    \begin{align*}
        -2b^3\equiv0~(mod~4)
    \end{align*}
   $\implies b^3\equiv0~(mod~4)\implies b=0,2~(mod~4)$.\\ 
   Then $b=4k$ or $b=4k+2$.\\
    And by comparing the imaginary part, we get
    \begin{align*}
        \frac{d^2}{2}+cd=\frac{a^2b}{2}+\frac{ab^2}{2}+\frac{b^3}{8}+\frac{b^3D}{8}+\frac{A_1b}{2}+\frac{A_2a}{2},
    \end{align*}
    which gives\\
    \begin{align*}
        4d^2+8cd=4a^2b+4ab^2+3b^3+b^3D+4A_1b+4A_2a.
    \end{align*}
    For $b=4k$, we get
    \begin{align*}
        d^2+2cd=4a^2k^2+4^2k^2a+3\times4^2k^3+4^2k^3D+4A_1k+A_2a
    \end{align*}
    gives,
    \begin{equation}\label{A2}
        d^2\equiv A_2a~(mod~2).
    \end{equation}
    Since the discriminant of \eqref{Elliptic curve1} is non-zero, we have
    \begin{align*}
        4\left(A_1+A_2\left(\frac{1+\sqrt{D}}{2}\right)\right)^3+27\left(B_1+B_2\left(\frac{1+\sqrt{D}}{2}\right)\right)^2\neq 0
    \end{align*}
    \begin{align*}
       0&\neq 4A_1^3+\frac{A_1^2A_2}{2}+\frac{A_1^2A_2\sqrt{D}}{2}+\frac{A_1A_2^2}{2}+\frac{A_1A_2^2\sqrt{D}}{2}+\frac{A_2^3+DA_2^3+2A_2^3\sqrt{D}}{4}\\&+A_1^2A_2+A_1^2A_2\sqrt{D}+\frac{A_1A_2^2}{2}+\frac{DA_1A_2^2}{2}+A_1A_2^2\sqrt{D}+27B_1^2+\frac{27B_2^2}{4}+\frac{27DB_2^2}{4}\\&+\frac{27B_2^2\sqrt{D}}{2}B_1B_2+B_1B_2\sqrt{D},
    \end{align*}
    which gives,
    \begin{align*}
        2A_1^2A_2+2A_1A_2^2+A_3^2+2A_1A_2^2-6A_1A_2^2+3B_2^2-9DB_2^2\equiv 0~(mod~4),
    \end{align*}
    so,
    \begin{align*}
        2A_1^2A_2-4A_1A_2^2+A_2^3+2A_1A_2^2-6B_2^2\equiv0~(mod~2).
    \end{align*}
    $\implies A_2^3\equiv 0~(mod~2)\implies A_2\equiv 0~(mod~2)\implies A_2$ is even. \\
    Therefore, from \eqref{A2} we can conclude that $d$ is even. One can proceed the same way for $b=4k+2$. Hence our claim and we have the following theorem.
\end{proof}

\begin{theorem}\label{thm1}
    If $x,y\in \mathbb Q(\sqrt{D})$ is of finite order in $E({\mathbb Q(\sqrt{D}))}$, which are not $2$-torsion defined over $\mathbb{Q}(\sqrt{D})$, then $x,y\in \mathbb Z(\sqrt{D})$, where $D=-1, -2, -3,-7,-11$.
\end{theorem}
\begin{rmk}
    If we have $y=0$, that is $c=0=d$ then we have an elliptic curve over $\mathbb{Q}$, and that corresponds to the classical case. For example one may consider the following;\\
    Consider the elliptic curve $x^3-a^2x$ over $\mathbb{Q}(\sqrt{D})$, let $D=-7$. One can choose  $a$ such that $a\in \mathcal{O}_{\mathbb{Q}(\sqrt{D})}$ and $a^2\not\in \mathbb{Z}(\sqrt{D})$. Then $E$ will have the $2$-torsion $(a,0)$.
\end{rmk}

Now we are looking for non-ED case, i.e., for $D=-19, -43, -67, -163.$\\

\subsection{Extended Nagell-Lutz for non-ED}
\begin{definition}
     Let $\alpha,\beta\in{\mathcal{O}_K}$, we say that $\alpha$ is divisible by $\beta$, if  $\beta\in <\alpha>$.
\end{definition}
\begin{rmk}
    Since we are working on an imaginary quadratic extension of class number $1$, unique factorization is guaranteed. So, the above definition is well defined. Through out the paper, we are using the unique factorization property. 
\end{rmk}
\begin{rmk}
    Let \begin{equation*}
         E:y^2=x^3+Ax+B ~~with~~ A,B\in \mathbb Q(\sqrt{D}).
     \end{equation*}\\
     Then actually $A, B\in{\mathcal{O}_K}$.
\end{rmk}
\begin{lemma}\label{lemma7}
Let
\begin{equation*}
         E:y^2=x^3+Ax+B ~~with~~ A,B\in {\mathcal{O}_K}.
\end{equation*}
Then for any $(x,y)\in E(\mathbb Q(\sqrt{D}))$, $den(x)\in<p>$ if and only if $den(y)\in <p>$.
\end{lemma}
\begin{proof}
     Let $x=\frac{x_1}{p^rx_2}$, $y=\frac{y_1}{p^sy_2}$, where $x_i,y_i\in{\mathcal{O}_K}~and~x_1x_2\notin <p>, y_1y_2\notin <p>$.\\
     If $den(x)\in<p>$ then $r>0$, which gives
     \begin{align*}
         \frac{y_1^2}{p^{2s}y_2^2}&=\left(\frac{x_1}{p^rx_2}\right)^3+A\left(\frac{x_1}{p^r y_2}\right)+B\\
         &=\frac{x_1^3+Ap^{2r}x_1x_2^2+Bp^{3r}x_2^3}{p^{3r}x_2^3},
     \end{align*}
     since $ x_1\notin <p>$\\
     \begin{align*}
         x_1^3+Ap^{2r}x_1x_2^2+Bp^{3r}x_2^3\notin<p>.
     \end{align*}
    Hence we get
    \begin{align*}
        2s=3r \implies s>0,
    \end{align*}
     $\implies den(y)\in<p> $. Converse is also true.\\
    Therefore, $2s=3r\in\mathbb Z\implies s=3q,r=2q,q\in \mathbb Z$.
\end{proof}

Let $E:y^2=x^3+Ax+B,~~A,B\in{\mathcal{O}_K}$.\\
If $y\neq 0$, $E\setminus \{(x,0)\}$ can be transformed into $E^{'}:s=t^3+Ats^2+Bs^3$, where $s=\frac{1}{y}$ and $t=\frac{x}{y}$.\\
Therefore we can define a map, 
\begin{align*}
    \phi:E\setminus \{(x,0)\}\longrightarrow E^{'},
\end{align*} 
by,
\begin{align*}
    (x,y)\longrightarrow(t,s),~~~\infty \longrightarrow(0,0).
\end{align*}

\begin{lemma}\label{lemma8}
     Let $E:y^2=x^3+Ax+B$ with $A,B\in {\mathcal{O}_K}$. Any torsion point $(x,y)\in E(\mathbb Q(\sqrt{D}))$ of order $2$ has $x\in{\mathcal{O}_K}$ and $y=0$.
\end{lemma}
 \begin{proof}
     Let $P=(x,y)$ has order $2$.\\
    In other words,
    \begin{align*}
        2P=\infty,
    \end{align*}
    which gives
    \begin{align*}
        P=-P.
    \end{align*}
    \\
    So, we can conclude that $y=0$.\\
    Therefore $x$ is a root of $x^3+Ax+B=0,~~~ A,B\in{\mathcal{O}_K}$.\\
    Since, $P\in E(\mathbb Q(\sqrt{D}))$, we will get $x\in\mathbb Q(\sqrt{D})$.\\
    and
    \begin{align*}
        x=\frac{a+b\sqrt{D}}{c+d\sqrt{D}}.
    \end{align*}
    So from our Elliptic curve,
    \begin{align*}
        0=\left(\frac{a+b\sqrt{D}}{c+d\sqrt{D}}\right)^3+A\left(\frac{a+b\sqrt{D}}{c+d\sqrt{D}}\right)+B.
    \end{align*}
    Let $a+b\sqrt{D}=a, c+d\sqrt{D}=b$, then
    \begin{align*}
        0&=\left(\frac{a}{b}\right)^3+A\left(\frac{a}{b}\right)+B\\
        &=a^3+Aab^2+Bb^3.
    \end{align*}
    So, $a^3\in<b>$, but we have $<a,b>=<c>$ (since we are working on PID and $a$ and $b$ are relatively prime).\\
    Which gives,
    \begin{align*}
        <a,b>=<1>.
    \end{align*}
    We have,
    \begin{align*}
        <b>=<a^3,b>\supseteq <a,b>^3=<1>,
    \end{align*}
    which implies $b$ is a unit in $\mathbb Q(\sqrt{D})$.\\
    Therefore,
    \begin{align*}
        x=\frac{a}{b}\in {\mathcal{O}_K}.
    \end{align*}
 \end{proof}

\begin{definition}
    For any $x\neq 0\in \mathbb Q(\sqrt{D})$, $p$-adic value of $x$ is
    \begin{align*}
        g_p(x)=g_p(\frac{a}{b})=r,
    \end{align*}
    where\\
    \begin{align*}
       \frac{a}{b}=p^r\frac{a_1}{b_1},~a,~b,~a_1,~b_1\in{\mathcal{O}_K},~a_1b_1\notin <p>. 
    \end{align*}
    
    Therefore by Lemma \ref{lemma1}, we have,
    \begin{align*}
        g_p(x)=-2q, g_p(y)=-3q.
    \end{align*}
    Define\\
    \begin{align*}
         E_r=\{(x,y)\in E(\mathbb Q(\sqrt{D})):g_p(x)\leq -2r, g_p(y)\leq -3r\}\bigcup \{\infty\}.
    \end{align*}
   
    \smallskip
    Clearly, $E(\mathbb Q(\sqrt{D}))\supseteq E_1\supseteq E_2\supseteq ...$
\end{definition}
\begin{lemma}\label{lemma9}
$(x,y)\in E_r$ if and only if $s\in <p^{3r}>$ and if $s\in <p^{3r}>$ then $t\in <p^r>$.
\end{lemma}
\begin{proof}
    If $(x,y)\in E_r$, then by definition $g_p(y)\leq -3r$.\\
    So we have,
    \begin{equation*}
        y=\frac{y_1}{p^{3r}y_2},~~~~y_1\notin <p>.
    \end{equation*}
Clearly,
\begin{equation*}
    s=\frac{1}{y}=\frac{p^{3r}y_2}{y_1}
\end{equation*}
\begin{align*}
    \implies s\in  <p^{3r}>.
\end{align*}
Conversely, suppose $s\in<p^{3r}>$ then $den(y)\in <p^{3r}>$. By Lemma \ref{lemma1}, $den(x)\in<p^{2r}>$.
\begin{align*}
    \implies (x,y)\in E_r.
\end{align*}
If $s\in <p^{3r}>\implies (x,y)\in E_r$, so the exact power of $p$ in $den(y)$ is $p^{3k}$ for some $k\geq r$.\\
By Lemma \ref{lemma1}, $p^{2k}$ is the exact power of $p$ in $den(x)$.\\
We have $t=\frac{x}{y}\implies t\in <p^k>$. Therefore, $t\in <p^{r}>$.

\end{proof}

\begin{lemma}\label{lemma10}
    Any vertical line $t=k$, where $k$ is a constant such that $k\in <p>$, intersects $E^'$ in atmost one point $(t,s)$ with $s\in <p>$. This line is not tangent to this point of intersection. 
\end{lemma}
\textbf{Note:} $a\equiv 0(mod~ p)$, if $a\in <p>$. Similarly for any power of $p$.
\begin{proof}
Suppose, the line intersects $E^'$ at $2$ points $(t,s_1),(t,s_2)\in E^'$ such that $s_i\in <p>$, $s_1\equiv s_2\equiv 0(mod~p)$.\\
Write $s_i=ps_i^'$.
Suppose that $s_1-s_2\in <p^k>$ for some $k\geq 1$. That is,
\begin{align*}
    s_1\equiv s_2(mod~p^k)~~~~~~for~~some~~ k\geq 1.
\end{align*}

Which gives,
    \begin{align*}
        s_1^{'}\equiv s_2^{'}(p^{k-1}).
    \end{align*}
    Then\\
    \smallskip
    \begin{align*}
        s_1^2\equiv s_2^2(p^{k+1}).
    \end{align*}
    Similarly,
    \begin{align*}
         s_1^3\equiv s_2^3(p^{k+2}).
    \end{align*}
    Therefore,
    \begin{align*}
        s_1=k^3+Aks_2^2+Bs_1^3\equiv k^3+Aks_2^2+Bs_2^3\equiv s_2(p^{k+1}).
    \end{align*}
    So by induction $s_1\equiv s_2(p)$ for all $n\geq 1.$ So $s_1=s_2$.\\
    Slope of the tangent line to the curve is 
    \begin{align*}
        \frac{ds}{dt}&=3t^2+As^2+2Ast\frac{ds}{dt}+3Bs^2\frac{ds}{dt}\\
        &=\frac{3t^2+As^2}{1-2Ast-3Bs^2}.
    \end{align*}
    If the line $t=k$ is tangent to the curve at $(s,t)$\\
    $\implies$\begin{align*}
        1-2Ast-3Bs^2=0,
    \end{align*} 
    but $s\equiv t\equiv 0(p)$, so
    \begin{align*}
        1-2Ast-3Bs^2\equiv 1. 
    \end{align*}
    Therefore, $t=k$ is not tangent to the curve.
    
\end{proof}
\begin{lemma}\label{lemma11}
    Suppose, the line $s=\alpha t+\beta$ intersects the curve at the points $(t_1,s_1)$ and $(t_2,s_2)$ in $E^{'}$. Then
\begin{align*}
    \alpha=\frac{t_2^{2}t_1 t_2+t_1^{2}+As_2^{2}}{1-A(s_1+s_2)t_1-B(s_2^{2}+s_1 s_2+s_1^{2})}.
\end{align*}
\end{lemma}
\begin{proof}
    Suppose $t_1\neq t_2$ then $\alpha=\frac{s_2-s_1}{t_2-t_1}$.\\
    We have $s_i=t_i^3+At_is_i^2+Bs_i^3$,\\
    which gives,
    \begin{align*}
        (s_2-s_1)(1-A(s_1+s_2)t_1-B(s_2^2+s_1s_2+s_1^2))&=(s_2-s_1)-A(s_2^2-s_1^2)t_1-B(s_2^3-s_1^3)\\
        &=(s_2-As_2^2t_2-Bs_2^3)-(s_1-As_1^2t_1-Bs_1^3)+As_2^2(t_2-t_1)\\
        &=t_2^3-t_1^3+As_2^2(t_2-t_1)\\
        &=(t_2-t_1)(t_2^2+t_1 t_2+t_1^2+As_2^2).
    \end{align*}
    So,
    \begin{equation}\label{beta}
        \frac{s_2-s_1}{t_2-t_1}=\alpha=\frac{t_2^2+t_1 t_2+t_1^2+As_2^2}{1-A(s_1+s_2)t_1-B(s_2^2+s_1s_2+s_1^2)}.\\
   \end{equation}
    If $t_1=t_2$, then by Lemma \ref{lemma 4} both points are identical. By differentiation we will get the tangent line as
    \begin{align*}
        \frac{ds}{dt}=\frac{3t^2+As^2}{1-2Ast-3Bs^2},
    \end{align*}
    which is same as our expression when $t=t_1=t_2$ and $s=s_1=s_2$.
\end{proof}

\begin{lemma}\label{lemma12}
    Let $E_r^{'}=\phi(E_r)$ and $E^{'}(\mathbb Q(\sqrt{D}))$. Then $E_r^{'}$ is a subgroup of $E^{'}(\mathbb Q(\sqrt{D}))$.
\end{lemma}
\begin{proof}
    Let $(x,y)\in E_r,\phi(x,y)=(t,s)$.\\
    We have $s\in <p^{3r}>,t\in <p^r>$, also 
    \begin{equation*}
        \alpha=\frac{t_2^2+t_1 t_2+t_1^2+As_2^2}{1-A(s_1+s_2)t_1-B(s_2^2+s_1s_2+s_1^2)}
    \end{equation*}
    $\implies \alpha\in <p^{2r}>$.
    Since $\beta=s-\alpha t\implies \beta\in<p^{3r}>$.\\
For $p_1,p_2\in E_r^{'}$ and $p_1+p_2=p_3$, let $s=\alpha t+\beta$ be the line passing through $p_1$ and $p_2$,\\ $\implies$
    \begin{align*}
        \alpha t+\beta=t^3+At(\alpha t+\beta)^2+B(\alpha t+\beta)^3,
    \end{align*}
    then,
    \begin{align*}
        0=t^3(1+\alpha^2 A+\alpha^3 B)+t^2(2A\alpha \beta+3\alpha^2\beta B)+\ldots
    \end{align*}
    let $p^*=(t^*,s^*)$ be the new intersection point. Then $p_3=(t_3,s_3)=-p^*$.\\$\implies$
    \begin{align*}
        t_1+t_2-t_3=t_1+t_2+t^*=-\frac{2A\alpha\beta+2B\alpha^2\beta}{1+\alpha^2A+\alpha^3 B}\equiv 0(p^{5r}).
    \end{align*}
    Since $t_i\in <p^r>$ for $i=1,2\implies  t_3\in<p^r>$. Also, $s_3=\alpha t_3+\beta\equiv 0(p^{3r})$. Then, by the Lemma \ref{lemma7}, $\pm p_3\in E_r^{'}\implies E_r^{'}$ is a subgroup of $E^{'}(\mathbb Q(\sqrt{D}))$.
 \end{proof} 

    For $P=(t,s)\in E_r^'$, write $t(P)=t$. Then $t(P_1)+t(P_2)=t(P_3)(mod~p^{5r})$. Suppose $P\in E_r^'$ has order $m$, assume that $m\notin <p>$. Then there exists $r>0$ such that $P\in E_r$ but $P\notin E_{r+1}$. If $m\notin<p>$ then
    \begin{align*}
        t(mP)=mt(P)(mod~p^{5r}),
    \end{align*}
    which gives,
    \begin{align*}
        0=mt(P)(mod~p^{5r}),
    \end{align*}
    that implies, $t(P)\in <p^{5r}>$, which is a contradiction.\\
    Therefore, $E_1^'$ has no point of finite order. So, using Proposition \ref{proposition} and Lemma \ref{lemma12} we have the following theorem.

\begin{theorem}\label{thm2}
    If $(x,y)\in E(\mathbb Q(\sqrt{D}))$ is of finite order, which are not $2$-torsion defined over $\mathbb{Q}(\sqrt{D})$, then $x,y\in \mathbb Z(\sqrt{D})$, where $D=-19, -43, -67, -163$.
\end{theorem}
Therefore, from Theorem \ref{thm1} and \ref{thm2} we have proved our main Theorem \ref{main thm}.

\section{Application of Nagell-Lutz}
In this section, we obtain a certain application of the Nagell-Lutz theorem in terms of producing elements of class groups of real quadratic fields. The idea is as follows:

Consider an elliptic curve $E$ given by the equation
$$y^2=x^3+Ax+B$$
over $\mathbb{Q}(\sqrt{-d})$ for the specific values of $d$ such that the corresponding ring of integers is UFD.

\subsection{Torsion of a family of elliptic curves over $\mathbb Q(i)$}
 Calculate the torsion of the elliptic curve  $E$  given by $y^2=x^3+(a+ib)x$. By the Nagell-Lutz theorem proved in the previous section, we have
$$y^2|4.(a+ib)^3$$
So, the possibility is that
$$y^2=(a+ib)^2, 4, 4(a+ib)^2$$
Write
$$y=y_1+iy_2$$
So, we have 
$$(y_1+iy_2)^2=(a+ib)^2$$
Substituting the above into the equation of $E$ we have
$$y^2=x^3+(a+ib)x$$
which can be written as
$$(a+ib)^2=x^3+(a+ib)x$$
Taking the norm, we have 
$$N(a+ib)^2=N(x)N(x^2+a+ib)$$
From this we have 
$$p^2=(x_1^2+x_2^2)N(x^2+a+ib)\;.$$
here $x=x_1+ix_2$ and $a^2+b^2=p$. Hence, either $p=x_1^2+x_2^2$ or $p^2=x_1^2+x_2^2$ or $x_1^2+x_2^2=1$. 
In the first case, we have 
$$x_1^2+x_2^2\equiv0 \mod p$$
So, all possibilities are 
$$X+Y\equiv0 \mod p$$
where $X,Y$ both are squares mod $p$. Since $p\equiv 1$ mod 4 we have $2p-1$ solutions. In addition, in this case, we need to satisfy the following.
$$N(x^2+a+ib)=p$$
which means that $(x_1^2-x_2^2+a)^2+(2x_1x_2+b)^2=p$. So modulo $p$ we have 
$$(x_1^2-x_2^2+a)^2+(2x_1x_2+b)^2=0$$
Expanding, we get 
$$x_1^4+x_2^4+a^2-2x_1^2x_2^2+2ax_1^2-2ax_2^2+4x_1^2x_2^2+b^2+4x_1x_2b=0$$
hence
$$x_1^4+x_2^4-2x_1^2x_2^2+2ax_1^2-2ax_2^2+4x_1^2x_2^2+4x_1x_2b=0$$
as $a^2+b^2=p$.
The above expression simplified is 
$$(x_1^2+x_2^2)^2+2ax_1^2-2ax_2^2+4x_1x_2b=0$$
Now using the fact that $x_1^2+x_2^2=0 \mod p$ we have 
$$2a(x_1^2-x_2^2)+4x_1x_2b=0$$
This implies
$$a(x_1^2-x_2^2)+2x_1x_2b=0$$
Now using $a=\pm {ib} \mod p$ here $i^2=-1 \mod p$, 
$$a(x_1^2-x_2^2+2ix_1x_2)=0$$
or $$a(x_1^2-x_2^2-2ix_1x_2)=0$$
From the above, we have
$$2x_2^2-2ix_1x_2\equiv 0 \mod p$$
or 
$$2x_1^2-2ix_1x_2\equiv 0\mod p$$
Using these,
$$x_2=0 ~or~ x_2=ix_1$$
or 
$$x_1=0, x_1=ix_2$$
Combining these, we have
$$x_1^2+x_2^2=p=x_1^2-x_2^2$$
Hence $2x_2^2=0$, implying $x_2=0$ and hence $x_1^2=p$, which is not possible, since $p$ is a prime.

Hence, we have
$$x_1=0, x_2^2=p, x_2=0, x_1^2=p$$
this is also not possible.
So, no solution in this case.
It can be possible that $$x_1^2+x_2^2=p^2$$
then
$$(x_1^2+x_2^2)^2+2ax_1^2-2ax_2^2+4x_1x_2b=1$$
modulo p.
Substituting $x_1^2+x_2^2=p^2$ mod $p$ we have
$$p^4+2a(x_1^2-x_2^2)+4x_1x_2b=1$$
So modulo $p$ we have 
$$a(x_1^2-x_2^2+2ix_1x_2)=1$$
or
$$a(x_1^2-x_2^2-2ix_1x_2)=1$$
Substituting $x_1^2=-x_2^2 \mod p$, we have 
$$2x_1^2+2x_1(ix_2)=1$$
or $$2x_1^2-2ix_1^2=1$$
modulo $p$.
In the first  case,
$$4x_1^2=1$$
or 
$$0=1$$
The second sub case is absurd.
In the first sub case, we have 
$$p|4x_1^2-1$$
So, $p|2x_1+1$ or $p|2x_1-1$
Therefore, we have 
$$x_1=\frac{pk-1}{2}$$
or $$x_1=\frac{pk+1}{2}$$
Consequently, we have 
$$x_2=\frac{pk+1}{2}$$
or 
$$x_2=\frac{pk-1}{2}\;.$$
For a unique $k$ , in this case equal to $1$. So we have 2-solutions here.

 It can also happen that 
$$x_1^2+x_2^2=1$$
and
$$(x_1^2+x_2^2)^2+2ax_1^2-2ax_2^2+4x_1x_2b=p^2$$
So we have 
$$1+2a(x_1^2-x_2^2)+4x_1x_2b=0 \mod p$$
$$1+2a(x_1^2+x_2^2-2x_2^2)+4x_1x_2b=0 \mod p$$
$$1+2a(1-2x_2^2)+4iax_1x_2=0 \mod p$$
or 
$$1+2a(1-2x_2^2)-4iax_1x_2=0$$
Then we have 
$$(4ix_1x_2+4x_2^2-2)a=1$$

Substituting $x_2=0,x=\pm{1}$ or the other way.

We have
$$-2.a=1$$
So we have $(p-2)a=1$, which is not possible. So, no solutions. Therefore, the subgroup obtained from this is $\mathbb Z/3\mathbb Z$.
 But Now observe that 
$$y^2=x(x^2+a+ib)$$
gives 
$$N(y^2)=N(x)N(x^2+a+ib)$$
So we have $\mathbb Z/3$ in the torsion subgroup of $E$. In addition, it is easy to see that there is a 2-torsion on this curve $E$. So, it can be $\mathbb{Z}/6\mathbb Z$. It can be at most $\mathbb Z/6\mathbb Z\oplus \mathbb Z/2\mathbb Z$ according to the results of Najman in \cite{N}.

\section{An interesting application to produce elements in the class group of quadratic fields}

In this section we prove the following theorem:

\begin{theorem}
Given an elliptic curve $y^2=x^3+(-2m+b\sqrt{-d})x$ over $\mathbb Q(\sqrt{-d})$, we have a homomorphism from the group of 2-torsion points on this elliptic curve to the class group of the number field $\mathbb Q(\sqrt{-p^3+4mp})$ for infinitely many primes $p$, not dividing $2m$ and the class group contains a subgroup isomorphic to $\mathbb Z/2\mathbb Z\times \mathbb Z/2\mathbb Z$. Here, $m$ is a positive integer.
\end{theorem}

\begin{proof}

Let $E$ be an elliptic curve defined over $\mathbb Q(\sqrt{-d})$, where $d>0$ and square-free an integer. Consider the elliptic curves of the form 
$$y^2=x^3+(a+b\sqrt{-d})x$$
then applying the norm map on both sides we get
$$N(y)^2=N(x)(N(x^2+a+b\sqrt{-d}))$$
The above expression 
$$N(x^2+a+b\sqrt{-d})$$
is nothing but
$$(x_1^2+dx_2^2)^2+2a(x_1^2-dx_2^2)+4bdx_1x_2$$
where $x=x_1+ix_2$.
The part 
$$2a(x_1^2-dx^2)+4bdx_1x_2$$
can be written as
$$2a(x_1^2+dx_2^2)-4adx_2^2+4bdx_1x_2$$
which is 
$$2aN(x)-4dx_2(ax_2-bx_1)\;.$$
Therefore, if we have 
$ax_2-bx_1=0$ then we get
$$N(y)^2=N(x)^2+2aN(x)\;.$$
Or an elliptic curve of the form
$$Y^2=X^3+2aX$$
Now, the condition means $a/b=x_1/x_2$, that is, $(a,b)$ and $(x_1,x_2)$ gives the same line when joined from $(0,0)$. Then we have $$x_1+\sqrt{-d}x_2=c(a+b\sqrt{-d})\;.$$ Substituting in the equation of $E$, we get 
$$y^2=c^3(a+b\sqrt{-d})^3-c(a+b\sqrt{-d})^2$$
that is
$$y^2=c(a+b\sqrt{-d})^2(c^2(a+b\sqrt{-d})-1)$$
Now suppose that $y=0$, then we have 
$$c^2(a+b\sqrt{-d})^2=0$$
or $$c^2(a+b\sqrt{-d})=1$$
which gives $ac^2=1$
so $a=1$, hence we have the 2-torsion subgroup of $E$, $\mathbb Z/2\mathbb Z\times \mathbb Z/2\mathbb Z$. Applying the norm, this group gives the 2-torsion subgroup of 
$$Y^2=X^3+2aX$$
Let $Y=0$, then we have $X=0$ or $X^2=-2a$, so we have $X=N(x)$, so we have 
$$(x_1^2+dx_2^2)^2=-2a$$
and $$x_1+dx_2^2=\sqrt{-2a}$$
if we take $a=-2m^2$ for a positive integer $m$, then we get the
$$x_1^2+dx_2^2=2m$$
So $N(x)=2m$. Hence , $(0,0), (2m,0)$ are two points in the torsion group of $Y^2=X^3+2aX$. Now applying the technique due to Soleng \cite{S} we find that, starting from the above elliptic curve we can obtain a quadratic field of the form $\mathbb Q(\sqrt{-p^3-2ap})$ for infinitely many primes $p$, such that the above field has a torsion subgroup isomorphic to $\mathbb Z/2\mathbb Z\times \mathbb Z/2\mathbb Z$, obtained from the 2-torsion subgroup of $Y^2=X^3+2aX$, where $a=-2m^2$. Hence, starting from $E$ the actual elliptic curve , applying the norm and then applying the Soleng's technique in \cite{S} we can achieve a family of imaginary quadratic fields whose class group has a $2$-torsion subgroup isomorphic to $\mathbb Z/2\mathbb Z\times \mathbb Z/2\mathbb Z$.

Let us recall in brief the construction due to Soleng \cite{S}. If $(x,y)$ is a point in $Y^2=X^3+2aX$, then $(x+p,y)$ is a point in
$$Y^2=X^3+(3p^2-a)X-p^3-2ap$$
We only need to assume that $(x+p,y)$ is a primitive point on the new elliptic curve. For that we need 
$gcd(x+p, y, (x+p)^2+(3p^2-a))=1$. Now take $(x,y)=(0,0)$ or $(x,y)=(2m,0)$. Then we have to show that 
$$(p,3p^2-2m)=1$$
and 
$$(2m+p, 3p^2-2m)=1$$
In the first case , if we assume that $2m$ is not a multiple of $p$, then we are done. In the second case, we have 
$$\alpha(2m+p)+\beta(3p^2-2m)=1$$
for some integers $\alpha, \beta$.
Then, we can write
$$(\alpha-\beta)2m+(\alpha+3p\beta)p=1$$
So, we need to assume that $2m$ is not a multiple of $p$ as before. Hence, we can guaranty primitive points in this way. Since there are infinitely many primes $p$, such that $2m$ is not divisible by $p$, this is guaranteed.
\end{proof}

\section{Acknowledgment}

The authors thank the SRM University AP for hosting this project. The authors also thank Prof. Kalyan Chakraborty for many valuable suggestions regarding the paper. The authors also thank Prof. John Cremona for discussions and suggestions of improvements to the first arxiv version of the preprint. The authors are grateful to Prof. Filip Najman for his valuable suggestions on the main theorem, which led to an improvement of the paper.


\end{document}